\documentclass{lms}
\usepackage{amsmath,amsfonts,amssymb}
\usepackage[style=alphabetic, sorting=nyt, firstinits=true]{biblatex}
\usepackage{mathtools}
\usepackage[affil-it]{authblk}
\bibliography{ref}
\renewbibmacro{in:}{\ifentrytype{article}{}{\printtext{\bibstring{in}\intitlepunct}}}

\renewcommand{\emph}[1]{\textbf{\textit{#1}}}
\newcommand{\F}{\mathbb{F}}
\newcommand{\Z}{\mathbb{Z}}
\newcommand{\N}{\mathbb{N}}
\newcommand{\G}{\Gamma}

\DeclareMathOperator{\rank}{rank}

\DeclareMathOperator{\tr}{Tr}

\DeclareMathOperator{\chr}{char}
\DeclareMathOperator{\spn}{span}
\DeclareMathOperator{\diam}{diam}
\DeclareMathOperator{\lcm}{lcm}

\newenvironment{en}
{\begin{enumerate}}
{\end{enumerate}}

\newenvironment{eq}
{\begin{displaymath}}
{\end{displaymath}}

\newnumbered{de}{Definition}[section]
\newtheorem{thm}[de]{Theorem}
\newtheorem{prop}[de]{Proposition}
\newtheorem{lem}[de]{Lemma}
\newtheorem{cor}[de]{Corollary}
\newnumbered{rem}[de]{Remark}
\newnumbered{exm}[de]{Example}
\newnumbered{conj}[de]{Conjectures}

\title[Diameter of Finite Simple Groups]{A Diameter Bound for Finite Simple Groups of Large Rank}
\author{Arindam Biswas
\footnote[1]{Laboratoire de Math\'{e}matiques, B\^{a}timent 425, Facult\'{e} des Sciences d'Orsay, Universit\'{e} Paris-Sud XI, 91405 Orsay Cedex, France. Email: arindam.biswas@math.u-psud.fr.}, 
Yilong Yang
\footnote[2]{3170 Sawtelle Blvd Apt 203, Los Angeles, CA 90066, USA. Email: yy26@math.ucla.edu\vspace{-13pt}}}
\date{\today}
\classno{20G40, 20F69 (primary), 05C25, 20E32, 51N30, 11N13, 05C81 (secondary)}

\begin{document}
\maketitle
\begin{abstract}
Given a non-abelian finite simple group $G$ of Lie type, and an arbitrary symmetric generating set $S$, it is conjectured by L\'{a}szl\'{o} Babai that its Cayley graph $\G(G,S)$ will have a diameter bound of $(\log |G|)^{O(1)}$. However, little progress has been made when the rank of $G$ is large. In this article, we shall show that if $G$ has rank $n$, and its base field has bounded size, then the diameter of $\G(G,S)$ would be bounded by $\exp(O(n(\log n)^3))$.
\end{abstract}
\section{Introduction}
\subsection{History and Background}

Given a group $G$ and a symmetric generating set $S$, one can construct a corresponding Cayley graph $\G(G,S)$. Its vertices are elements of $G$, and two vertices $g,h\in G$ are connected by an edge iff there is an element $s\in S$ such that $sg=h$. The Cayley graph is a metric space where the distance between two vertices is simply the length of the shortest path from one to the other. In this way, we can discuss the diameter of the Cayley graph. Equivalently, we can also define the diameter to be the smallest number $\ell$ such that every element of $G$ can be written as a product of at most $\ell$ elements of $S$.

If $G$ is a non-abelian finite simple group, we expect all its Cayley graphs to have good connectivity. In particular, we have the following conjecture of Babai:

\begin{conj}[Babai, \cite{BS92}]
For any non-abelian finite simple group $G$, and for any symmetric generating set, the diameter of the Cayley graph is bounded by $(\log |G|)^{O(1)}$, where the implied constant is absolute.
\end{conj}

The first class of simple groups verified for Babai's conjecture was $\mathrm{PSL}_{2}(\Z/p\Z)$ with $p$ prime, by Helfgott \cite{H08}. Afterwards, a lot of research was done on the diameters and related expansion properties of these Cayley graphs.

The best result to date are those by Pyber and Szabo \cite{PS10}, and Breuillard, Green and Tao \cite{BGT11}. They verified Babai's conjecture for all finite simple groups of Lie type with bounded rank.\footnote[3]{The preprint of PS was published on arXiv in 2010, and proved Babai's conjecture for all finite simple groups of Lie type with bounded ranks. Unlike PS, BGT first announced  these results only for  finite simple groups not belonging to the Suzuki and Ree family. Their method also applies to these cases, but this only appeared later in \cite{BGTsuz11} and \cite{BGGT15}. We thank L\'{a}szl\'{o} Pyber for this remark.}

For all non-abelian finite simple groups, Breuillard and Tointon \cite{BT15} also obtained a diameter bound of $\max(|G|^\epsilon,C_\epsilon)$ for arbitrary $\epsilon>0$ and a constant $C_\epsilon$ depending only on $\epsilon$. The diameter bounds in all these previous results depend poorly on the rank of the group. It is one of the aims of this paper to improve the dependency on the rank, in the case of finite simple groups of Lie type.

On the other hand, a lot of research was also done for the symmetric group $\mathrm{S}_n$ and the alternating group $\mathrm{A}_n$. In 1988, Babai and Seress showed the following theorem.

\begin{thm}[Babai and Seress, \cite{BS88}]
Let $G= \mathrm{S}_{n}$ or $G=\mathrm{A}_{n}$, then for any symmetric generating set, $$\diam(G)\leq \exp(\sqrt{n\log n}(1+o_n(1))=\exp(\sqrt{\log|G|}(1+o_n(1)).$$
\end{thm}

This was the best known bound for $\mathrm{S}_{n}$ or $\mathrm{A}_{n}$ for over two decades, until Helfgott and Seress recently showed the following.

\begin{thm}[Helfgott and Seress, \cite{HS14}]
Let $G=\mathrm{S}_{n}$ or $G=\mathrm{A}_{n}$, then for any symmetric generating set, $$\diam(G)\leq \exp(O((\log n)^{4}\log \log n))=\exp((\log \log |G|)^{O(1)})$$.
\end{thm}

In this article we give a modest upper bound on the diameter for finite simple groups of Lie type, where the dependency on rank is lessened.

\begin{thm}[Main Theorem]
Let $G$ be a finite simple group of Lie type, with rank $n$ and base field $\F_q$, then
$$\diam(G)\leq q^{O(n((\log n+\log q)^3))}.$$
In particular, if the base field has bounded size, we have
$$\diam(G)\leq \exp(O(n(\log n)^3))=\exp(O(\sqrt{\log |G|}(\log\log |G|)^3)).$$
\end{thm}

\subsection{Preliminaries}

\begin{de}
Given an algebraic group $G$ over a finite field $\F_q$, the \emph{algebraic rank} of $G$ is the dimension of a maximal torus in $G$.
\end{de}

\begin{prop}
There is an absolute constant $C$, such that any finite simple group of Lie type $G$ of algebraic rank larger than $C$ must be a projective special linear group $\mathrm{PSL}_n(\F_q)$, a projective symplectic group $\mathrm{PSp}_n(\F_q)$, a projective special unitary group $\mathrm{PSU}_n(\F_q)$, or the simple quotient $\mathrm{P\Omega}_n(\F_q)$ of the derived subgroup $\mathrm{\Omega}_n(\F_q)$ of the orthogonal group $\mathrm{O}_n(\F_q)$.
\end{prop}
\begin{proof}
Going through the list of finite simple groups of Lie type, the algebraic ranks of all but the above four families of groups are bounded by an absolute constant $C$.
\end{proof}

In this paper, we are only interested in finite simple groups of Lie type with large ranks. Therefore, these four families of groups are all we need to deal with.

\begin{prop}
The algebraic rank of groups $\mathrm{GL}_n(\F_q)$, $\mathrm{SL}_n(\F_q)$, $\mathrm{PSL}_n(\F_q)$, $\mathrm{Sp}_n(\F_q)$, $\mathrm{PSp}_n(\F_q)$, $\mathrm{U}_n(\F_q)$, $\mathrm{SU}_n(\F_q)$, $\mathrm{PSU}_n(\F_q)$, $\mathrm{O}_n(\F_q)$, $\mathrm{\Omega}_n(\F_q)$, $\mathrm{P\Omega}_n(\F_q)$ are between $\frac{n}{2}-1$ and $n$.
\end{prop}
\begin{proof}
This is done by computing their algebraic rank one by one.
\end{proof}

Let $G$ be any group. Given a subset $S$ of $G$, we shall use $S^d$ to denote $\{g_1...g_d:g_1,...,g_d\in S\}$, i.e., the set of product of $d$ elements of $S$.

We shall use the ``big O'' and ``small o'' convention. We let $O(g(x))$ to denote a quantity whose absolute value is bounded by $Cg(x)$ for some absolute constant $C$. And we let $o_n(1)$ to denote a quantity whose limit is 0 as $n$ goes to infinity.

\subsection{Outline of the Proof}

\begin{de}
Given a matrix $A$, we define its \emph{degree} to be $\deg(A)=\rank(A-I)$. Equivalently, it is $n$ minus the dimension of the fixed subspace of $A$.
\end{de}

The general idea is to find a matrix of small degree and close to the identity. And from there on, the expansion would be very fast. This is analogous to \cite{BS88}, the case of symmetric and alternating groups, where one first finds elements of small support (i.e., 2-cycles or 3-cycles), and then proceed to fill up the whole group.

Given any finite simple group of Lie type, we proceed via the following steps. Let $G$ be $\mathrm{SL}_n(\F_q)$, $\mathrm{Sp}_n(\F_q)$, $\mathrm{SU}_n(\F_q)$ or $\mathrm{\Omega}_n(\F_q)$. Then we can assume that $G$ acts on a vector space $V$ of dimension $n$.

\begin{en}
\item We start by finding a subset of small diameter in $G$ near the identity, which we call a $t$-transversal set. Such a set contains an extension for every linear or isometric embeddings of any $t$-dimensional subspace $W$ into $V$. This is an analogy to a $t$-transitive subset of a symmetric group $\mathrm{S}_n$. See Section~\ref{sec:ttrans} for the linear case, and Section~\ref{sec:ousp} for the symplectic case, the unitary case, and the orthogonal case. 
\item Using the $t$-transversal set above, we can find a special matrix, which we call a P-matrix. Some large power of a P-matrix will have a much smaller degree than the original one. This is the main degree reducing step, inspired by Lemma 1 of the paper of Babai and Seress \cite{BS88}. The result on P-matrices is dependent on an inequality of primes, which is deduced in Section~\ref{sec:NumTh}. For P-matrices, see Section~\ref{sec:pmat}.
\item Combined with repeated commutators with carefully chosen elements from the $t$-transversal set, we can repeat the above step many times, until we reach a matrix of very small degree. See Section~\ref{sec:comm} for the linear case, and Section~\ref{sec:ouspcomm} for the symplectic case, the unitary case, and the orthogonal case.
\item From a small degree matrix, one can quickly fill up its conjugacy class. Then by the result of Liebeck and Shalev \cite{LS01}, from a conjugacy class, one can quickly fill up the whole group. See Section~\ref{sec:conj}.
\end{en}

\subsection{Acknowledgements} 

We would like to thank Emmanuel Breuillard and Terence Tao for a number of helpful discussions and advice in this subject.

\section{$t$-Transversal Sets with Small Diameters}
\label{sec:ttrans}

Let $V$ be a vector space of dimension $n$ over the field $\F_q$. Let the group $\mathrm{GL}_n(\F_q)$ act on it naturally.

\begin{de}
A subset $S$ of $\mathrm{GL}_n(\F_q)$ is called a \emph{$t$-transversal set} if given any embedding $X$ of a $t$-dimensional subspace $W$ into $V$, we can find $A\in S$ that extends $X$ on $W$.
\end{de}

\begin{lem}
$\mathrm{GL}_n(\F_q)$ is $t$-transversal for all $t$, and $\mathrm{SL}_n(\F_q)$ is $t$-transversal for all $t<n$.
\end{lem}
\begin{proof}
Let $W$ be any subspace with a basis $w_1,...,w_t$. We can complete this into a basis of $V$ with new vectors $v_1,...,v_{n-t}$. Let $A$ be a matrix with column vectors $w_1,...,w_t,v_1,...,v_{n-t}$. In the case when $t<n$, we can multiply $v_{n-t}$ by a constant so that $\det (A)=1$.

For any embedding $X$ of $W$ into $V$, $X(w_1),...,X(w_t)$ are linearly independent. We can complete this into a basis of $V$ with new vectors $u_1,...,u_{n-t}$. Let $B$ be a matrix with column vectors $X(w_1),...,X(w_t),u_1,...,u_{n-t}$. In the case when $t<n$, we can multiply $u_{n-t}$ by a constant so that $\det (B)=1$.

Now $B(A)^{-1}$ is in $\mathrm{GL}_n(\F_q)$ and, if $t<n$, also in $\mathrm{SL}_n(\F_q)$. We also have $(B(A)^{-1})|_W=X$.
\end{proof}

\begin{lem}
For any symmetric subset $S$ of $\mathrm{GL}_n(\F_q)$, if the subgroup generated by $S$ is $t$-transversal, then $\bigcup_{d=1}^{d=q^{nt}} S^d$ is $t$-transversal.
\end{lem}
\begin{proof}
Let $W$ be any $t$-dimensional subspace. Let $L(W)$ be the set of embeddings of $W$ into $V$. Let $H$ be the subgroup generated by $S$. Then an element $g$ of $H$ acts on $L(W)$ by $g(X)=(g\circ X)|_W$ for any $X\in L(W)$. Let $\G$ be the corresponding Schreier graph of this action of $H$ on $L(W)$ with generating set $S$, i.e., the vertices are elements of $L(W)$, and two vertices $X,Y$ are connected iff $g(X)=Y$ for some $g\in S$.

Now, since $H$ is $t$-transversal, the graph $\G$ is connected. So the diameter of $\G$ is trivially bounded by its number of vertices, which is at most $q^{nt}$. As a result, the set $\bigcup_{d=1}^{q^{nt}} S^d$ is $t$-transversal.
\end{proof}

\begin{cor}
\label{Cor:ttrans}
Given any symmetric generating set $S$ for $\mathrm{GL}_n(\F_q)$, the set $\bigcup_{d=1}^{d=q^{nt}} S^d$ is $t$-transversal. If $t<n$, then the same statement is true with $\mathrm{SL}_n(\F_q)$ replacing $\mathrm{GL}_n(\F_q)$.
\end{cor}

\section{An Inequality on Primes}
\label{sec:NumTh}

In this section, we shall establish an inequality on primes to be used in the next section. 

Throughout this section, we shall fix a prime $p_0$ and fix a power of it $q_0$, which in the next section shall become the characteristic and the order of a finite field.

Let $p_1,...,p_r$ be the first $r$ primes coprime to $p_0(q_0-1)$. Let $M$ be the least common multiple of $p_1-1,p_2-1,p_3-1,...,p_r-1$. Let $S$ be the sum of $p_1-1,p_2-1,p_3-1,...,p_r-1$. Our goal for this section is the following proposition:

\begin{lem}
\label{Lem:PrimeLCM}
There exist absolute constants $c_1$ and $c_2$ such that, if $p_r\geq c_1\log q_0$, then $$S\leq (p_r)^2\leq c_2(\log M)^3.$$
\end{lem}

Before we prove this, let us first set up more notations. Let $P^+$ be the function that sends each positive integer to its largest prime factor. Let $P=\{p_1,...,p_r\}$. For any $\delta>0$, let $P_\delta=\{\text{prime number }p:3\leq p\leq p_r,P^+(p-1)\geq (p_r)^\delta\}$, and let $P_\delta^*=P_\delta\cap P$.

We start by citing an important theorem of Fouvry.

\begin{lem}[(Fouvry \cite{F85})]
There is an absolute constant $\delta>\frac{2}{3}$, and an absolute constant $c_0, c_p$, such that for $p_r\geq c_p$, we have
$$|P_\delta|\geq c_0\frac{p_r}{\log p_r}.$$
\end{lem}

\begin{cor}
There is an absolute constant $\delta>\frac{2}{3}$, and absolute constants $c_0$ and $c_3$, such that $|P_\delta^*|\geq c_0\frac{p_r}{\log p_r}-c_3\frac{\log q_0}{\log\log q_0}$ if $q_0>e^e$, and $|P_\delta^*|\geq c_0\frac{p_r}{\log p_r}-3$ if $q_0\leq e^e$. 
\end{cor}
\begin{proof}
By prime number theorem, when $\log\log q_0>1$ (i.e., $q_0>e^e\approx 15.15$), the number of prime factors of $p_0(q_0-1)$ is bounded by $c_3\frac{\log q_0}{\log\log q_0}$ for some absolute constant $c_3$. If $q_0\leq e^e$, then there are at most 3 prime factors of $p_0(q_0-1)$.
\end{proof}

\begin{lem}
Let $p\geq (p_r)^\delta$ be some prime. Then $$|((P^+)^{-1}(p)+1)\cap P_\delta^*|\leq \frac{2p_r}{(\log 2) (p_r^\delta -1)}.$$
\end{lem}
\begin{proof}

We assume that the left hand side of the inequality is non-zero, because otherwise the inequality is trivial. Then $p$ divides some $p_i-1$, which is not a prime. So in particular, $2p< p_r$.

The set $((P^+)^{-1}(p)+1)\cap P_\delta^*$ is contained in the set of primes $\leq p_r$ that are congruent to 1 mod $p$. By the Brun-Titchmarsh theorem, combined with the fact that $p\geq (p_r)^\delta$, we have
\begin{align*}
|(P^+)^{-1}(p)\cap P_\delta^*|\leq& \frac{2p_r}{\phi(p)\log\frac{p_r}{p}}\\
\leq& \frac{2p_r}{(\log 2) (p-1)}\\
\leq& \frac{2p_r}{(\log 2) ((p_r)^\delta-1)}.
\end{align*}

Here $\phi$ is the Euler totient function.
\end{proof}

Now we have enough to prove Lemma~\ref{Lem:PrimeLCM}.

\begin{proof}[of Lemma~\ref{Lem:PrimeLCM}]
The first inequality is straightforward
$$S\leq\sum_{\text{prime }p\leq p_r}p\leq (p_r)^2.$$

All the primes in $P^+(P_\delta^*-1)$ are factors of $M$, and they are all larger than $(p_r)^\delta$. Furthermore, when $q_0>e^e$, we have
\begin{align*}
|P^+(P_\delta^*-1)|\geq&\frac{|P_\delta^*|}{\max _{p\geq (p_r)^\delta} |((P^+)^{-1}(p)+1)\cap P_\delta^*|}\\
\geq& (c_0\frac{p_r}{\log p_r}-c_3\frac{\log q_0}{\log\log q_0})/(\frac{2p_r}{(\log 2) ((p_r)^\delta-1)})\\
\geq& \frac{\log 2}{2}((p_r)^\delta -1) (\frac{c_0}{\log p_r}-\frac{c_3}{p_r}\frac{\log q_0}{\log\log q_0}).
\end{align*}

So, if $p_r>c_1\log q_0$ for an absolute constant $c_1$ such that $c_3\frac{1+\log c_1}{c_1}<\frac{c_0}{2}$, then we have
\begin{align*}
\log M\geq& |P^+(P_\delta^*-1)|\log((p_r)^\delta)\\
\geq& \frac{\log 2}{2}\delta((p_r)^\delta -1) (c_0-c_3\frac{\log q_0}{\log\log q_0}\frac{\log p_r}{p_r})\\
\geq& \frac{\log 2}{2}\delta((p_r)^\delta -1) (c_0-c_3\frac{1+\log c_1}{c_1})\\
\geq& \frac{\log 2}{4}\delta c_0((p_r)^\delta -1).
\end{align*}

Since $\delta>\frac{2}{3}$, we can pick some constant such that $c_2(\log M)^3\geq (p_r)^2$.

Now, suppose $q_0\leq e^e$. Then similarly we have
\begin{align*}
|P^+(P_\delta^*-1)|\geq&\frac{|P_\delta^*|}{\max _{p\geq (p_r)^\delta} |((P^+)^{-1}(p)+1)\cap P_\delta^*|}\\
\geq& (c_0\frac{p_r}{\log p_r}-3)/(\frac{2p_r}{(\log 2) ((p_r)^\delta-1)})\\
\geq& \frac{\log 2}{2}((p_r)^\delta -1) (\frac{c_0}{\log p_r}-\frac{3}{p_r}).
\end{align*}

So, if $p_r>c_1\log q_0$ for a sufficiently large absolute constant $c_1$ such that $\frac{3\log p_r}{p_r}<\frac{c_0}{2}$, then we have
\begin{align*}
\log M\geq& |P^+(P_\delta^*-1)|\log((p_r)^\delta)\\
\geq& \frac{\log 2}{2}\delta((p_r)^\delta -1) (c_0-\frac{3\log p_r}{p_r})\\
\geq& \frac{\log 2}{4}\delta c_0((p_r)^\delta -1).
\end{align*}

Since $\delta>\frac{2}{3}$, we can again pick some constant such that $c_2(\log M)^3\geq (p_r)^2$.

\end{proof}

As a side note, for any improved value of $\delta$ in the Fouvry's theorem, our diameter bound in this paper would improve to $q^{O(n(\log n+\log q)^\frac{2}{\delta})}$ for finite simple groups of Lie type of rank $n$ over $\F_q$.

If one were to assume the Hardy-Littlewood conjucture on prime tuples, the $\delta$ could be improved to $1-o(1)$. Combine this with the more efficient estimate $S\leq \frac{(p_r)^2}{\log p_r}$, the diameter bound of this paper would improve to $q^{O(n(\log n+\log q)^2)}$ for finite simple groups of Lie type of rank $n$ over $\F_q$.

\section{P-Matrices and Degree Reduction}
\label{sec:pmat}

This section aims to show that, given a P-matrix, we can reduce its degree by raising it to a large power.

\begin{de}
Let $\F_q$ be a finite field of characteristic $p$, and let $p_1,p_2,...,p_r$ be the first $r$ primes coprime to $p(q-1)$. Then a matrix $A$ over $\F_q$ is called a \emph{P($r$)-matrix} if, for each $i\leq r$, it has a primitive $p_i$-th root of unity in the algebraic closure of $\F_q$ as an eigenvalue.
\end{de}

\begin{lem}
Let $A$ be a matrix over $\F_q$, a field with characteristic $p$. Let $m$ be any number coprime to $p$. Then if $A$ has a primitive $m$-th root of unity as an eigenvalue, $A$ must have degree at least $o_m(q)$. Here $o_m(q)$ denotes the multiplicative order of $p$ in $\Z/m\Z$.
\end{lem}
\begin{proof}
Let $\Phi_m(X)$ be the $m$-th cyclotomic polynomial. Then by Galois theory over finite field, the polynomial $\Phi_m(X)$ factors into distinct irreducible polynomials of degree $o_m(q)$. See, e.g., \cite{SL}.

Therefore, if $A$ has a primitive $m$-th root of unity as an eigenvalue, then $A$ must have at least $o_m(q)$ primitive $m$-th roots of unity as eigenvalues, and the result follows.
\end{proof}

\begin{lem}
\label{Lem:FieldExt}
Let $A$ be a non-identity matrix over finite field $\F_q$ of characteristic $p$. Suppose $A$ has a primitive $m$-th root of unity as an eigenvalue in the algebraic closure of $\F_q$. Let $P(m)$ be the set of all prime divisors of $m$ coprime to $p(q-1)$. Then $A$ has degree at least $\lcm_{x\in P(m)}(x-1)$. Here $\lcm$ denote the least common multiple.
\end{lem}
\begin{proof}
By replacing $m$ by its factor, we can assume that $m$ is the product of primes in $P(m)$. In particular, we can assume that $m$ is coprime with $p$. Then $A$ has degree at least $o_m(q)$.

For each $x\in P(m)$, since $x$ is a prime, $o_x(q)=x-1$. Since $o_x(q)$ divides $o_m(q)$ whenever $x$ divides $m$, therefore $x-1$ is a factor of $o_m(q)$ for all $x\in P(m)$.

So the degree of $A$ is at least $o_m(q)\geq \lcm_{x\in P(m)}(x-1)$.
\end{proof}

\begin{lem}
\label{Lem:pmatrix}
Let $n$ be an integer, and let $q$ be a power of the prime $p$. Then we can find an integer $r$ and an absolute constant $c$, such that the following is true:
\begin{en}
\item If $p_1,p_2,...,p_r$ are the first $r$ primes coprime to $p(q-1)$, then $\lcm_{i=1}^r(p_i-1)>n^4$, and $\sum_{i=1}^r(p_i-1)< c(\log n+\log q)^3$.
\item Let $A\in \mathrm{GL}_n(\F_q)$ where the field has characteristic $p$, and $\deg A=k$. If $A$ is a P($r$)-matrix, then there exists $\ell\in\N$ such that $A^\ell$ will be a non-identity matrix of degree at most $\frac{k}{4}$, and 1 is the only eigenvalue of $A^\ell$ lying in $\F_q$.
\end{en}
\end{lem}
\begin{proof}
\text{ }

\emph{The First Statement:} 

Let $M$ be the least common multiple, and let S be the sum. Let $c_1$ be the constant as in Lemma~\ref{Lem:PrimeLCM}.

Pick $p_r$ to be the smallest prime such that $M>n^4$ and $p_r>c_1\log q$. Then the second condition guarantees that $S<c_2(\log M)^3$, according to Lemma~\ref{Lem:PrimeLCM}.

Now, if $p_r\leq 2c_1\log q$, then for some absolute constant $c_4$ by the Prime Number Theorem, we have
\begin{align*}
\log M\leq& \sum_{i=1}^r \log p_r\\
\leq& c_4p_r\\
\leq& 2c_1c_4\log q.
\end{align*}

So $S\leq c(\log q)^3$ for some absolute constant $c$.

Suppose $p_r>2c_1\log q$. Then by the Bertrand-Chebyshev Theorem, $p_{r-1}>c_1\log q$. Let $M'=\lcm_{i=1}^{r-1}(p_i-1)$. Then by the minimality of $p_r$, we must have $M'\leq n^4$. In particular, we have
\begin{align*}
4\log n\geq&\log M'\\
\geq& (\frac{p_{r-1}}{\sqrt{c_2}})^\frac{2}{3}.
\end{align*}

So, we have $p_{r-1}\leq 8\sqrt{c_2}(\log n)^{\frac{3}{2}}$. Then $p_r\leq 16\sqrt{c_2}(\log n)^{\frac{3}{2}}$.

Furthermore, we have
\begin{align*}
\log M\leq& \log(M' p_r)\\
\leq& \log(n^4(16\sqrt{c_2}(\log n)^{\frac{3}{2}}))\\
<& 6\log n +\log(16\sqrt{c_2}).
\end{align*}

So $S< c_1(\log M)^3 < c(\log n)^3$ for some absolute constant $c$.

\emph{The Second Statement:}

Let $M_i$ denote the least common multiple of $p_1-1,p_2-1,...,p_i-1$. Let $t_1=p_1-1$ and $t_i=\frac{M_i}{M_{i-1}}$ for $i>1$. Then $\prod_{i=1}^r t_i=M_r>n^4$.

Let $N=\{1,2,...,n\}$. Let $d_1,...,d_n$ be the eigenvalues of $A$ in the algebraic closure of $\F_q$.

For each $j\in N$, let $P_j$ be the set of prime factors of the multiplicative order of $d_j$ among $p_1,...,p_r$. Then by Lemma~\ref{Lem:FieldExt}, for each $j\in N$,
\begin{eq}
\prod_{p_i\in P_j}t_i\leq \lcm_{p_i\in P}(p_i-1)\leq k.
\end{eq}

Now let $n(i)$ denote the number of $P_j$ that contain $p_i$.

We take the weighted average $T$ of these $n(i)$ with weight $\log t_i$. The sum of the weights is $\sum_{i=1}^r\log t_i>4\log n$.

\begin{align*}
T&=\frac{\sum_{1\leq i\leq r}n(i)\log t_i}{\sum\log t_i}\\
&=\frac{\sum_{1\leq i\leq r}\sum_{j\in N}(\log t_i) 1_{p_i\in P_j}}{\sum\log t_i}\\
&=\frac{\sum_{j\in N}\sum_{1\leq i\leq r}(\log t_i) 1_{p_i\in P_j}}{\sum\log t_i}\\
&\leq\frac{\sum_{j\in N}\sum_{p_i\in P_j}\log t_i}{4\log n}\\
&\leq \frac{k\log k}{4\log n}\\
&\leq\frac{k}{4}.
\end{align*}

So there is a $p_i$ such that $n(i)\leq\frac{k}{4}$. So if $A$ has order $m(A)$, then $A^{\frac{m(A)}{p_i}}$ is the desired non-identity matrix of degree at most $\frac{k}{4}$. Every eigenvalue of the latter matrix not equal to 1 is a primitive $p_i$-th root of unity, which would be outside of $\F_q$.
\end{proof}

\section{Commutators and Degree Reduction}
\label{sec:comm}

In this section, we shall use repeated commutators with elements of a $t$-transversal set. This way, we repeatedly create P-matrices and raise them to a large power, and would eventually end up with a matrix of very small degree.

\begin{de}
Given any element $g$ of a group $G$ and a symmetric generating set $S$ for $G$, the \emph{length} of $g$ is $\ell(g)=\min\{d\in\N:g=s_1...s_d\text{ for some }s_1,...,s_d\in S\}$.
\end{de}

\begin{prop}
For any matrices $A,B$, $\deg(ABA^{-1}B^{-1})\leq 2\min(\deg A,\deg B)$.
\end{prop}
\begin{proof}
\begin{align*}
\deg(ABA^{-1}B^{-1})&=\rank(ABA^{-1}B^{-1}-I)\\
&=\rank(AB-BA)\\
&=\rank((A-I)(B-I)-(B-I)(A-I))\\
&\leq\rank(A-I)(B-I)+\rank(B-I)(A-I)\\
&\leq\rank(A-I)+\rank(A-I)\\
&=2\rank(A-I).
\end{align*}
Similarly, we also have $\deg(ABA^{-1}B^{-1})\leq 2\rank(B-I)$. So we are done.
\end{proof}

\begin{lem}
\label{Lem:subexi}
Fix any matrix $A\in \mathrm{GL}(V)$ of degree $k$, such that the eigenvalues of $A$ are either 1 or outside of $\F_q$. For any $t\leq\frac{k}{2}$, we can find a subspace $W$ of $V$ with the following properties:
\begin{en}
\item $\dim W=t$;
\item $W\cap AW=\{0\}$.
\end{en}
\end{lem}
\begin{proof}
We shall proceed by induction on the dimension of $W$. Let $V_A$ be the subspace of fixed points of $A$ in $V$.

\emph{Initial Step:} Suppose $t=1$. Simply pick any vector $v$ outside of $V_A$, and let $W$ be the span of $v$. We have $W\cap V_A=\{0\}$ by choice of $v$. Since $A$ has no eigenvalue in $\F_q$ other than 1, $v$ and $Av$ must be linearly independent. So $W\cap AW=\{0\}$.

\emph{Inductive Step:} Suppose we have found a subspace $W$ of dimension $t-1$ such that $W\cap AW=\{0\}$. I claim that, when $t\leq\frac{k}{2}$, we can find another vector $v$, such that the desired subspace is the span of $v$ and $W$.

To prove the existence of $v$, let us count the number of vectors to avoid. We want $v$ to avoid $V_A+W+AW$. Afterwards, it is enough to let $Av$ avoid any linear combination of $v$ and $W+AW$. So we need $v$ to avoid $\bigcup_{x\in\F_q}(A-x)^{-1}(W + AW)$. Here we shall interpret $(A-x)^{-1}$ as the pullback map of subsets.

Now, since $A$ has no eigenvalue in $\F_q$ other than 1, therefore $A-x$ is invertible when $x\neq 1$. And $A-1$ has kernel exactly $V_A$, which has dimension $n-k$. So, we have
\begin{align*}
&|(V_A + W+AW) \cup (\bigcup_{x\in\F_q}(A-x)^{-1}(W + AW))|\\
\leq & q^{n-k+2t-2}+q^{n-k+2t-2}+(q-1)q^{2t-2}\\
< & q^{n-k+2t}.
\end{align*}

So as long as $2t\leq k$, we have $q^{n-k+2t}\leq q^n$. So it is possible to choose a vector $v$ as desired.
\end{proof}

\begin{lem}
\label{Lem:ordmatq}
Let $A$ be a matrix in the group $\mathrm{GL}_n(\F_q)$. Then the order of $A$ is less than $q^n$.
\end{lem}
\begin{proof}
By the rational canonical form of a matrix $A$ (see, e.g., \cite{H}), it is enough to prove the case when $A$ is a single rational jordan block. 

In this case, the characteristic polynomial $f(x)$ of $A$ is a power of an irreducible polynomial $g(x)$. Say $f=g^t$ and the characteristic of $\F_q$ is $p$. Then the order of $A$ is $(q^{\frac{n}{t}}-1)p^{\lceil\frac{\log t}{\log p}\rceil}<q^{\frac{n}{t}}p^{\lceil\frac{\log t}{\log p}\rceil}$.

Then it is enough to show that $p^{\lceil\frac{\log t}{\log p}\rceil}\leq q^{(1-\frac{1}{t})n}$. And since $p^{\lceil\frac{\log t}{\log p}\rceil}$ is the smallest $p$-power that is larger than or equal to $t$, and since $q$ is a power of $p$, it is enough to show that $t\leq q^{(1-\frac{1}{t})n}$. And this is true since $q\geq 2$ and $n\geq t$.
\end{proof}

\begin{prop}
\label{Prop:Bdddeg}
For any symmetric generating set $S$ of $\mathrm{GL}_n(\F_q)$, $\mathrm{GL}_n(\F_q)$ has a non-trivial element of degree at most $C(\log n+\log q)^3$ for some absolute constant $C$, of length less than $q^{C'n(\log n+\log q)^3}$ for some absolute constant $C'$. The same statement is true with $\mathrm{SL}_n(\F_q)$ replacing $\mathrm{GL}_n(\F_q)$.
\end{prop}
\begin{proof}
We pick $r$ and $c$ according to Lemma~\ref{Lem:pmatrix}. We may assume that $c(\log n+\log q)^3<n$, because otherwise the statement is trivial.

Let $p_1,...,p_r$ be the first $r$ primes coprime to $p(q-1)$. Let $f_i(x)$ be the irreducible polynomial over $\F_q$ for all the primitive $p_i$-th roots of unity, and let $C_i$ be the companion matrix of $f_i(x)$. 

\emph{Initial Step:}

Let us find our first P($r$)-matrix. Let $T$ be a $c(\log n+\log q)^3$-transversal set. Then by definition, we can find $A_0\in T$ that maps some subspace $W$ of dimension $c(\log n+\log q)^3$ onto itself, and that its restriction to this subspace is the matrix $(\bigoplus_{i=1}^rC_i)\oplus I$ for some arbitrary choices of basis on $W$, where $I$ is some identity matrix of suitable size.

In particular, $A_0$ is a P($r$)-matrix. Since $A_0\in T$, by choosing $T$ as in Corollary~\ref{Cor:ttrans}, $A_0$ has length bounded by $q^{cn(\log n+\log q)^3}$.

By using Lemma~\ref{Lem:pmatrix}, we can raise $A_0$ to a large power, and obtain a non-identity matrix $A_1$ of degree $\leq\frac{\deg(A_0)}{4}\leq\frac{n}{4}$, with eigenvalues either 1 or outside of $\F_q$. Since the order of $A_0$ is bounded by $q^n$ by Lemma~\ref{Lem:ordmatq}, the length of $A_1$ is bounded by $q^{cn(\log n+\log q)^3+n}$.

\emph{Inductive Step:}

Suppose we have obtained a non-identity matrix $A_j$ with eigenvalues either 1 or outside of $\F_q$, degree at most $\frac{n}{2^{j+1}}$, and length at most $q^{2cn(\log n+\log q)^3+j(n+2)}$. If $\deg A_j\leq 2c(\log n+\log q)^3$, then we stop. If not, then let us construct a non-identity matrix $A_{j+1}$ of even smaller degree.

First we shall transform $A_j$ into a P($r$)-matrix. Find a subspace $W_j$ of dimension at least $c(\log n+\log q)^3$ as in Lemma~\ref{Lem:subexi} using $A_j$. In particular, $W_j$ has trivial intersection with $A_jW_j$. Let $T'$ be a $2c(\log n+\log q)^3$-transversal set, then we can find $M_j\in T'$ that fixes $A_jW_j$, and restricts to a map from $W_j$ to $W_j$ as $\bigoplus_{i=1}^rC_i\oplus I$ for an arbitrary basis of $W_j$ and some identity matrix $I$ of suitable size.

Consider the commutator $M_jA_j^{-1}M_j^{-1}A_j$. Since $M_j$ fixes $A_jW_j$, we see that $M_jA_j^{-1}M_j^{-1}A_j$ restricted to $W_j$ is identical to $M_j$ restricted to $W_j$. 

In particular, $M_jA_j^{-1}M_j^{-1}A_j$ is a P($r$)-matrix, and it has degree at most $2\deg(A_j)$. Now we use Lemma~\ref{Lem:pmatrix} again, raising $M_jA_j^{-1}M_j^{-1}A_j$ to a large power, and we would obtain a matrix $A_{j+1}$ of degree at most $\frac{2\deg(A_j)}{4}$, with eigenvalues either 1 or outside of $\F_q$.

Since $M_j\in T'$, by choosing $T'$ as in Corollary~\ref{Cor:ttrans}, $M_j$ have length bounded by $q^{2cn(\log n+\log q)^3}$. And since the order of $M_jA_j^{-1}M_j^{-1}A_j$ is bounded by $q^n$ by Lemma~\ref{Lem:ordmatq}, the length of $A_{j+1}$ is at most
\begin{align*}
q^n(2q^{2cn(\log n+\log q)^3}+2q^{2cn(\log n+\log q)^3+j(n+2)})\leq q^{2cn(\log n+\log q)^3+(j+1)(n+2)}.
\end{align*}

We repeat the above induction $\frac{\log n}{\log 2}-1$ times, or stop early if we hit degree $2c(\log n+\log q)^3$. The last $A_j$ we obtained is the desired matrix of small degree and small length.
\end{proof}

\section{$t$-Transversal Sets for Orthogonal Groups, Symplectic Groups, and Unitary Groups}
\label{sec:ousp}

Let's fix some notation for the discussion of the following three sections. Let $V$ be a non-degenerate formed space of dimension $n$ over the field $\F_q$, with a non-degenerate quadratic form $Q$ (the orthogonal case), non-degenerate alternating bilinear form $B$ (the symplectic case), or non-degenerate Hermitian form $B$ with field automorphism $\sigma$ (the unitary case). In the orthogonal case, we shall let $B$ be the symmetric bilinear form obtained by polarizing $Q$, i.e., $B(v,w)=Q(v+w)-Q(v)-Q(w)$. Let $G$ be the group of isometries for $V$.

\begin{de}
\begin{en}
\item A vector $v\in V$ is \emph{singular} if $B(v,v)=0$ and (if applicable) $Q(v)=0$. 
\item A pair of singular vectors $v,w\in V$ is called a \emph{hyperbolic pair} if $B(v,w)=1$. 
\item The subspace generated by a hyperbolic pair is a \emph{hyperbolic plane}. 
\item A subspace $W$ of $V$ is \emph{anisotropic} if it contains no singular vector other than 0. 
\item A subspace is \emph{totally singular} if the form $B$ and (if applicable) the quadratic form $Q$ restricted to it is the zero form. 
\item Given any subspace $W$ of $V$, we define its \emph{orthogonal complement} to be $W^\perp:=\{v\in V:B(v,w)=0\text{ for all }w\in W\}$. Two subspaces $U,W$ of $V$ are \emph{orthogonal} if they are in each other's orthogonal complimant. We denote this as $U\perp W$.
\item The \emph{radical} of $V$ is $V^\perp$.
\item A subspace $W$ is \emph{radical-free} if $W\cap V^\perp=\{0\}$.
\end{en}
\end{de}

\begin{thm}[(Witt's Decomposition Theorem)]
The non-degenerate formed space $V$ has an orthogonal decomposition $V=V_{ani}\oplus(\bigoplus_{i=1}^mH_i)$, where $V_{ani}$ is anisotropic of dimension at most 2, and $H_i$ are hyperbolic planes. In particular, $V$ has a totally singular subspace of dimension at least $\frac{\dim(V)-2}{2}$, and any anisotropic space in $V$ has dimension at most $2$.
\end{thm}
\begin{proof}
See \cite{CGGA}.
\end{proof}

\begin{lem}
Recall that $V$ is a non-degenerate formed space.
\begin{en}
\item $V^\perp=\{0\}$ unless the non-degenerate form for $V$ is a quadratic form, and $\chr\F_q=2$.
\item $V^\perp$ has dimension at most 1.
\item For any subspace $W$, $\dim W+\dim W^\perp$ is equal to $\dim V$ if $W$ is radical-free, and $\dim V+1$ if $W$ is not.
\item For any subspace $W$, $(W^\perp)^\perp=W+V^\perp$.
\item A totally singular subspace is always radical-free.
\end{en}
\end{lem}
\begin{proof}
See \cite{CGGA}.
\end{proof}

\begin{de}
A subset $S$ of $G$ is called a \emph{singularly $t$-transversal set} if, for any isometric embedding $X$ of a $t$-dimensional totally singular subspace $W$ into $V$, we can find $A\in S$ that extends $X$ on $W$.
\end{de}

\begin{lem}[(Witt's Extension Lemma)]
$G$ is a singularly $t$-transversal set for any $t$.
\end{lem}
\begin{proof}
This is a special case of Witt's extension lemma, which states that any bijective isometry of radical-free subspaces of $V$ could be extended to an isometry of the whole formed space. See \cite{CGGA} for a proof.
\end{proof}

Now, since our focus is on the finite simple groups, we don't really use the full isometry group $G$. Rather, we are interested in its commutator subgroup $G'$.

\begin{lem}
For any $t\leq\frac{n-2}{5}$, the commutator subgroup $G'$ of $G$ is singularly $t$-transversal.
\end{lem}
\begin{proof}
Let $W$ be a totally singular space of dimension $t$. Let $X:W\to V$ be any isometric embedding from $W$ to $V$. 

\emph{Step 1:} I claim that there is a totally singular subspace $W'$, which is orthogonal to $W$ and $X(W)$, has trivial intersection with $W$ and $X(W)$, and has the same dimension as $W$.

To see this, we have $\dim W^\perp=\dim X(W)^\perp\geq n-t$. Therefore, $\dim (W^\perp\cap X(W)^\perp)\geq n-2t$. So in the subspace $W^\perp\cap X(W)^\perp$, we can find a subspace $W''$ of dimension $n-3t$ with trivial intersections with $W$ and $X(W)$. Now, since $W''$ is a formed space (possibly degenerate), it has a totally singular subspace of dimension at least $\frac{\dim W'' -2}{2}=\frac{n-3t-2}{2}\geq t$. So, from this totally singular space, we could simply pick any totally singular subspace of dimension $t$ to be the desired $W'$.

\emph{Step 2:} 

Let $Y:W\to W'$ be any bijective linear map. Since both spaces are totally singular, $Y$ is an isometry. So we could find an extension $A\in G$.

Let $Z:W\oplus W'\to X(W)\oplus W'$ be the linear map that restricts to $X$ on $W$, and restricts to the identity map on $W'$. Then by our choice of $W'$, this is a well-defined isometry of totally singular subspaces, and it would have an extension $B\in G$.

Consider $BA^{-1}B^{-1}A\in G'$. This would restrict to $X$ on $W$. So we are done.
\end{proof}

\begin{lem}
Let $S$ be any subset of $G$. If the subgroup generated by $S$ is singularly $t$-transversal, then $\bigcup_{d=1}^{d=q^{nt}} S^d$ is singularly $t$-transversal.
\end{lem}
\begin{proof}
Let $H$ be the subgroup generated by $S$. Let $W$ be any $t$-dimensional totally singular subspace, and let $L(W)$ be the set of isometric embeddings of $W$ into $V$. Then an element $g\in H$ acts on $L(W)$ by $g(X)=(g\circ X)|_W$ for any $X\in L(W)$. Let $\G$ be the corresponding Schreier graph of this action of $H$ on $L(W)$ with generating set $S$.

Any isometric embedding from $W$ to $V$ is a linear map. Therefore, there are at most $q^{nt}$ vertices for $\G$, where $t=\dim W$. And since $H$ is singularly $t$-transversal, the graph $\G$ must be connected. So $\G$ must have a diameter at most $q^{nt}$.
\end{proof}

\begin{cor}
\label{Cor:ttransOSU}
Given any symmetric generating set $S$ for $G$ or $G'$, the set $\bigcup_{d=1}^{d=q^{nt}} S^d$ is singularly $t$-transversal for $t\leq\frac{n-2}{5}$.
\end{cor}

\section{Degree Reducing for Orthogonal, Symplectic, Unitary Groups}
\label{sec:ouspcomm}

\begin{lem}
For any non-zero singular $v\in V$, there is a vector $w\in V$ such that $v,w$ form a hyperbolic pair. 
\end{lem}
\begin{proof}
Recall that $V$ is a non-degenerate formed space with an alternating bilinear, symmetric bilinear or Hermitian form $B$. In the case of characteristic 2, a symmetric bilinear form $B$ might be degenerate, even though the formed space itself is not. Let $\sigma$ be the field automorphism of the base field $F$ for the Hermitian form $B$, or identity if $B$ is bilinear. 

For any element $k\in F$, we define $\tr(x)=x+\sigma(x)$.

Now given a singular $v\in V$, since $V$ is a non-degenerate formed space, we can find a vector $w'\in V$ such that $B(v,w')\neq 0$. By scaling $w'$, we can assume that $B(v,w')=1$.

Suppose we can find an element $k\in F$ such that $\tr(k)=B(w',w')$, then $w=w'-kv$ is the desired vector forming a hyperbolic pair with $v$: $B(v,w)=B(v,w')=1$, and
\begin{align*}
&B(w'-kv,w'-kv)\\
=&B(w',w')-\tr(k)\\
=&0.
\end{align*}

Now, it remains to show that such $k$ always exists. 

Let $E$ be the subfield of $F$ fixed by $\sigma$. Then obviously $B(w',w')\in E$. So it is enough to show that either $E=\tr(F)$, or $B(w',w')=0$ for all $w'$.

Now, $\tr(F)$ is closed under addition, and it is also closed under multiplication by elements of $E$. So $\tr(F)$ is a $E$-vector space contained in $E$. So either $E=\tr(F)$, or $\tr(F)=0$.

In the case that $\tr(F)=0$, then $\sigma(x)=-x$ for all $x\in F$. But since $\sigma$ is a field automorphism, we must conclude that the field $F$ has characteristic 2, and $\sigma$ is the identity. Then the form $B$ is alternating, and $B(w',w')=0$ for any $w'\in V$.
\end{proof}

\begin{lem}
If a subspace $H$ of $V$ is an orthogonal sum of hyperbolic planes, then $H\cap H^\perp=\{0\}$.
\end{lem}
\begin{proof}
The subspace $H$ is an orthogonal sum of hyperbolic planes. Then let us assume that these planes are the linear span of hyperbolic pairs $(v_1,w_1),(v_2,w_2),(v_3,w_3),...,(v_t,w_t)$.

Suppose $v\in H\cap H^\perp$. Then for some scalars $a_i,b_i\in F$, we have
$$v=\sum_{i=1}^t a_i v_i + \sum_{i=1}^t b_i w_i.$$
Now, since $B(v,v_i)=0$, we can deduce that $b_i=0$. Similarly, since $B(v,w_i)=0$, we can deduce that $a_i=0$. So $v=0$.
\end{proof}

\begin{lem}
\label{Lem:sqorb}
Fix any nonzero elements $a,b,c\in\F_q$. Then the equation $ax^2+by^2+cz^2=0$ has a non-trivial solution in $\F_q$.
\end{lem}
\begin{proof}
If $\chr(\F_q)=2$, then $(\F_q)^*$ is a multiplicative group of odd order. So every nonzero element of $\F_q$ is a square.

Find $x,y,z$ such that $x^2=a^{-1}$, $y^2=b^{-1}$ and $z=0$. This is a non-trivial solution of the equation.

Suppose $q$ is odd. Let $S$ be the set of squares in $\F_q$. Then $|S|=\frac{q+1}{2}$. Then $|aS|+|-c-bS|>|\F_q|$. As a result, we have $aS\cap(-c-bS)\neq\varnothing$. So $-c\in aS+bS$.

Pick $x,y\in\F_q$ such that $ax^2+by^2=-c$. Then the triple $(x,y,1)$ is a non-trivial solution to the equation.
\end{proof}

\begin{lem}
\label{Lem:nullAtrans}
Fix any $A\in G$. Then given any totally singular subspace $W\in V$ of dimension $d$, we can find a subspace $W'$ of $W$ such that $W'$ is perperdicular to $AW'$, and $W'$ has dimension at least $\frac{d}{4}-\frac{3}{2}$.
\end{lem}
\begin{proof}
We proceed by induction on the dimension of $W$. 

\emph{Initial Step:} For the base case of the induction, suppose the dimension of $W$ is 7 or 8 or 9 or 10. Then all we need is to find a nonzero vector $v\in W$ such that $v\perp Av$. Suppose for contradiction that there is no such vector. 

Pick any non-zero $v_1\in W$. Let $W_1$ be the intersection of $W$ and $\spn\{v_1,Av_1,A^{-1}v_1\}^\perp$. Since $v_1$ is not perpendicular to $Av_1$, it is not in $\spn\{v_1,Av_1,A^{-1}v_1\}^\perp$. So $v_1\notin W_1$, and $W_1$ has dimension at least $\dim W -3\geq 4$. 

Pick any non-zero $v_2\in W_1$. Let $W_2$ be the intersection of $W_1$ and $\spn\{v_2,Av_2,A^{-1}v_2\}^\perp$. Then $W_2$ has dimension at least $\dim W_1-3\geq 1$ and similarly $v_2\notin W_2$. Pick any non-zero $v_3\in W_2$.

Now, we know $B(v_1,Av_1),B(v_2,Av_2),B(v_3,Av_3)$ are all in $\F_q^*$. We shall divide our discussion into two cases:

\emph{Orthogonal or Symplectic Case:} Let $a=B(v_1,Av_1),b=B(v_2,Av_2),c=B(v_3,Av_3)$. Then by Lemma~\ref{Lem:sqorb}, we can find a nontrivial triple $x,y,z\in\F_q$ such that $ax^2+by^2+cz^2=0$. Let $v=xv_1+yv_2+zv_3$, then we have
\begin{align*}
B(v,Av)=&x^2B(v_1,Av_1)+y^2B(v_2,Av_2)+z^2B(v_3,Av_3)\\
=&ax^2+by^2+cz^2\\
=&0.
\end{align*}

\emph{Unitary Case:} If $B$ is a Hermitian form for a field automorphism $\sigma$ of order 2, then let $F$ be the fixed subfield of $\sigma$. Let $N:\F_q\to F$ be the field norm, which is surjective.

Now, $\F_q$ is an $F$-vector space of dimension 2. So $B(v_1,Av_1),B(v_2,Av_2),B(v_3,Av_3)$ cannot be $F$-linearly independent in $\F_q$. So one can find non-trivial triple $a_1,a_2,a_3\in F$ such that $a_1B(v_1,Av_1)+a_2B(v_2,Av_2)+a_3B(v_3,Av_3)=0$.

Since the norm map is surjective, find $x_1,x_2,x_3\in \F_q$ such that $N(x_i)=a_i$. Let $v=x_1v_1+x_2v_2+x_3v_3$. Then we have
\begin{align*}
B(v,Av)=&N(x_1)B(v_1,Av_1)+N(x_2)B(v_2,Av_2)+N(x_3)B(v_3,Av_3)\\
=&a_1B(v_1,Av_1)+a_2B(v_2,Av_2)+a_3B(v_3,Av_3)\\
=&0.
\end{align*}

So in either case, we could find the desired non-trivial vector $v\in W$ such that $v\perp Av$.

\emph{Inductive Step:} Now let us proceed for general $W$ of larger dimension. Since the dimension of $W$ is at least 7, by the argument in the base case of the induction, we can find $v_1\in W$ such that $B(v_1,Av_1)=0$. Let $W_1$ be the intersection of $W$ and $\spn\{v_1,Av_1,A^{-1}v_1\}^\perp$. Then $W_1$ has dimension at least $d-3$. Pick any subspace $W_2$ of $W_1$ linearly independent from $v_1$. Then $W_2$ has dimension at least $d-4$ and at most $d-1$. Then by induction hypothesis, we can find $W_2'$ a subspace of $W_2$, such that $W_2'$ is perpendicular to $AW_2'$, and $W_2'$ has dimension at least $\frac{d-4}{4}-\frac{3}{2}$. 

Let $W'$ be the span of $W_2'$ and $v_1$. Then $W'$ will be perpendicular to $AW'$, and has dimension at least $\frac{d-4}{4}-\frac{3}{2}+1=\frac{d}{4}-\frac{3}{2}$. So we are done.
\end{proof}

\begin{lem}
\label{Lem:nullAlin}
Fix any $A\in G$ where all eigenvalues of $A$ are outside of $\F_q$. Then there is a $t$-dimensional totally singular subspace $W$ of $V$ such that $W\cap AW=\{0\}$, for any $t\leq \frac{n}{6}$.
\end{lem}
\begin{proof}

Fix $n$, which we assume to be at least 3, so that $V$ has at least one singular vector. We shall proceed by induction on the dimension of $W$.

\emph{Initial Step:} Suppose $t=1$. Simply pick any singular vector $v$, and let $W$ be the span of $v$. Since $A$ has no eigenvalue in $\F_q$, $v$ and $Av$ must be linearly independent. So $W\cap AW=\{0\}$.

\emph{Inductive Step:} Suppose we have found a totally singular subspace $W$ of dimension $t-1$ such that $W\cap AW=\{0\}$. I claim that, when $t\leq \frac{n}{6}$, we can find another singular vector $v$, such that the desired subspace is the span of $v$ and $W$.

First of all, we want $v$ to be a singular vector perpendicular to $W$. We know $W^\perp$ has dimension $n-t+1$, and by Witt's decomposition theorem, $V$ has a totally singular space of dimension at least $\frac{n-2}{2}$. This totally singular space will intersect $W^\perp$ in a subspace of dimension at least $\frac{n-2}{2}-t+1=\frac{n}{2}-t$. So there are at least $q^{\frac{n}{2}-t}$ singular vectors perpendicular to $W$.

Among these vectors, to prove the existence of a good $v$, we should count the number of vectors to avoid. We need $v$ to avoid $W+AW$. Afterwards, it is enough to have $Av$ avoiding any linear combination of $v$ and $W+AW$. To satisfy the second requirement, we need $v$ to avoid $\bigcup_{x\in\F_q}(A-x)^{-1}(W + AW)$. Here we shall interpret $(A-x)^{-1}$ as the pullback map of subsets.

Now, since $A$ has no eigenvalue in $\F_q$, therefore $A-x$ are all invertible. So, we have
\begin{align*}
&|(W+AW) \cup (\bigcup_{x\in\F_q}(A-x)^{-1}(W + AW))|\\
\leq & q^{2t-2}+q\times q^{2t-2}\\
< & q^{2t}.
\end{align*}

So as long as $2t\leq\frac{n}{2}-t$, i.e., $t\leq\frac{n}{6}$, then it is possible to choose a vector $v$ as desired.
\end{proof}

\begin{lem}
\label{Lem:ouspsubexi}
Fix any matrix $A\in G$ of degree $k$, such that the eigenvalues of $A$ are either 1 or outside of $\F_q$. Then we can find a subspace $W$ of $V$ with the following properties:
\begin{en}
\item $\dim W\geq\frac{k}{32}-\frac{7}{4}$
\item $W$ is totally singular;
\item $W\cap AW=\{0\}$;
\item $W\perp AW$.
\end{en}
\end{lem}
\begin{proof}
Let $V_A$ be the subspace of fixed points of $A$ in $V$. Let $V_r=V_A\cap(V_A)^\perp$. Choose any positive number $a$ to be determined later. Then either $V_r$ has dimension $<a$, or it has dimension $\geq a$.

\emph{Case of Large $V_r$:}

Suppose $V_r$ has dimension $\geq a$. Pick any non-zero singular $v_1\in V_r$, then we can find $w_1\in V$ such that $v_1,w_1$ form a hyperbolic pair. Let $V_{r1}$ be the intersection of $V_r$ with $\spn\{v_1,w_1\}^\perp$. Pick any non-zero singular $v_2\in V_{r1}$, then we can find $w_2$ in $\spn\{v_1,w_1\}^\perp$, such that $v_2,w_2$ form a hyperbolic pair. Then let $V_{r2}$ be the intersection of $V_{r1}$ with $\spn\{v_1,w_1,v_2,w_2\}^\perp$, and repeat. 

As long as $\dim V_{ri}>2$, then $V_{ri}$ cannot be anisotropic. So we can keep going at least $\lfloor\frac{a-2}{2}\rfloor$ times. Thus we obtained $w_1,...,w_{\lfloor\frac{a-2}{2}\rfloor}$. They span a totally singular space $W_r$ of dimension at least $\frac{a-3}{2}$. Then by Lemma~\ref{Lem:nullAtrans}, we can find a subspace $W$ of $W_r$, such that $W\perp AW$ and $W$ has dimension at least $\frac{a-3}{8}-\frac{3}{2}$.

I claim that, ignoring the dimension requirement, this $W$ satisfies all the desired properties. By construction of $W$, we have $W$ totally singular and $W\perp AW$. We only need to show that $W\cap AW=\{0\}$. 

For any vector $w=\sum_{i=1}^{\lfloor\frac{a-2}{2}\rfloor}a_iw_i\in W$, suppose it is perpendicular to $V_r$. Then for each $i$, since $B(v_i,w)=0$, we see that $a_i=0$. So $w=0$. To sum up, $W$ has trivial intersection with $(V_r)^\perp$. 

Suppose $w\in W\cap AW$. Then $w-A^{-1}w\in W$, and for any $v\in V_r$ we have
\begin{align*}
B(v,w-A^{-1}w)=&B(v,w)-B(v,A^{-1}w)\\
=&B(v,w)-B(Av,w)\\
=&B(v,w)-B(v,w)\\
=&0.
\end{align*}

So $w-A^{-1}w\in W\cap V_r^\perp=\{0\}$. So $w=Aw$, and $w\in W\cap V_A\subseteq W\cap (V_r)^\perp=\{0\}$.

To sum up, this $W$ is the space we desired, with dimension at least $\frac{a-3}{8}-\frac{3}{2}$.

\emph{Case of Small $V_r$:}

Suppose $V_r$ has dimension $<a$.

\emph{Step 1:} We want to first find a subspace $W_A$ of $(V_A)^\perp$ where $W_A\perp W_A$, and $W_A\cap AW_A=V_r$, and the codimension of $V_r$ in $W_A$ is at least $\frac{k-a-5}{6}$.

Now, the bilinear or Hermitian form $B$ restricted to $(V_A)^\perp$ is still bilinear or Hermitian, with exactly $V_r$ as the radical. So the space $V'=(V_A)^\perp/V_r$ has an induced bilinear or Hermitian form $B'$, and now $B'$ is non-degenerate. 

So $V'$ is a non-degenerate formed space with dimension at least $k-a$. Furthermore, since $V_r$ and $(V_A)^\perp$ are both $A$-invariant, $A$ induces a linear map $A'$ on $V'$. Clearly $A'$ has no non-trivial fixed point in $V'$, so all eigenvalues of $A'$ are ouside of $\F_q$. So by Lemma~\ref{Lem:nullAlin}, $V'$ has a totally singular subspace $W'$ of dimension at least $\lfloor\frac{k-a}{6}\rfloor\geq\frac{k-a-5}{6}$, such that $W'\cap A'W'=\{0\}$.

Let $W_A$ be the pullback of $W'$ through the projection map $(V_A)^\perp\to V'$. Since $W'$ is totally singular under $B'$, the form $B$ vanishes on $W_A$. (Note that in the orthogonal case, the quadratic form $Q$ might not vanish on $W_A$, so $W_A$ might not be totally singular.)

\emph{Step 2:} Now let us find a totally singular subspace $W_r$ of $W_A$ intersecting trivially with $V_r$ and has dimension at least $\frac{k-a-5}{6}$.  

If $\chr\F_q\neq 2$, or if we are not in the orthogonal case, then $W_A$ is totally singular. Pick any subspace $W_r$ of $W_A$ which has trivial intersection with $V_r$ and has dimension at least $\frac{k-a-5}{6}$, and we are done.

Suppose $\chr\F_q=2$, and we are in the orthogonal case, and $Q$ vanishes on $V_r$. Then the space $V'=(V_A)^\perp/V_r$ would have an induced non-degenerate quadratic form $Q'$ that corresponds to the non-degenerate bilinear form $B'$. So by Lemma~\ref{Lem:nullAlin}, when we picked $W'$ to be totally singular, we can pick it to be totally singular with respect to the quadratic form $Q'$. This way the subspace $W_A$ would be totally singular. Again pick any subspace $W_r$ of $W_A$ which has trivial intersection with $V_r$ and has dimension at least $\frac{k-a-5}{6}$, and we are done.

Finally, suppose now that $\chr\F_q=2$, and we are in the orthogonal case, and we have a vector $v_0\in V_r$ such that $Q(v_0)\neq 0$. Since $\chr\F_q=2$, the squaring map is bijective on $\F_q$, we can assume that $Q(v_0)=1$ by scaling $v_0$.

Define a map $X:W_A\to W_A$ such that $X(v)=v+\sqrt{Q(v)}v_0$. Here the square root is well defined because $\chr\F_q=2$. Then we have $Q(X(v))=0$ for all $v\in W_A$.

Furthermore, $X$ is linear. To see this, first we notice that for any $v,w$ in $W_A$, since $B$ vanishes on $W_A$,
\begin{align*}
Q(v+w)=Q(v)+Q(w)+B(v,w)=Q(v)+Q(w).
\end{align*} 

So we have
\begin{align*}
X(v+w)=&v+w+\sqrt{Q(v+w)}v_0\\
=&v+w+\sqrt{Q(v)+Q(w)}v_0\\
=&v+w+(\sqrt{Q(v)}+\sqrt{Q(w)})v_0\\
=&X(v)+X(w).
\end{align*}

For any scalar $a\in\F_q$, we also easily have $X(av)=aX(v)$.

Now, since $X$ is linear, $X(W_A)$ is a subspace of $W_A$. So $X(W_A)$ is a totally singular subspace.

Now pick any subspace $W_r$ of $W_A$ which has trivial intersection with $V_r$ and has dimension at least $\frac{k-a-5}{6}$. Then $X(W_r)$ is a totally singular subspace of $W_A$. It remains to show that this $X(W_r)$ intersects $V_r$ trivially and has the correct dimension.

For any vector $v$, if $X(v)\in V_r$, then $v=\sqrt{Q(v)}v_0+X(v)\in V_r$. So $X(W_r)$ only has trivial intersection with $V_r$. And since the kernel of $X$ is entirely in $V_r$, $X(W_r)$ has the same dimension as $W_r$.

So replace $W_r$ by $X(W_r)$, and we are done.

\emph{Step 3:} Now we construct the desired subspace $W$.

By Lemma~\ref{Lem:nullAtrans}, we find a subspace $W$ of $W_r$ such that $W\perp AW$, and $W$ has dimension at least  $\frac{k-a-5}{24}-\frac{3}{2}$.

I claim that, ignoring the dimension requirement, this $W$ satisfies all the desired properties.

First of all, we know $W$ is totally singular and $W\perp AW$. By construction, $W$ is in $(V_A)^\perp$ but has trivial intersection with $V_r$. Then since $W_A\cap AW_A=V_r$, we know that $W\cap AW=\{0\}$.

To sum up, this $W$ is the space we desired, with dimension at least $\frac{k-a-5}{24}-\frac{3}{2}$.

\emph{Find the Optimal $a$:}

Picking the optimal $a=\frac{k}{4}+1$ for both cases above, we eventually find the desired subspace $W$ of dimension at least $\frac{k}{32}-\frac{7}{4}$.
\end{proof}

\begin{prop}
\label{Prop:BdddegOSU}
For any symmetric generating set $S$ of $G$ or $G'$, there is a non-trivial element of degree at most $C(\log n+\log q)^3$ for some absolute constant $C$, of length less than $q^{C'n(\log n+\log q)^3}$ for some absolute constant $C'$.
\end{prop}
\begin{proof}
We pick $r$ and $c$ according to Lemma~\ref{Lem:pmatrix}. Let us assume that $2c(\log n+\log q)^3<\frac{n-2}{5}$, because otherwise the statement is trivial.

Let $p_1,...,p_r$ be the first $r$ primes coprime to $p(q-1)$. Let $f_i(x)$ be the irreducible polynomial over $\F_q$ for all the primitive $p_i$-th roots of unity, and let $C_i$ be the companion matrix of $\F_q$. 

\emph{Initial Step:}

Let us find our first P($r$)-matrix. Let $T$ be a singularly $c(\log n+\log q)^3$-transversal set. Let $W$ be any totally singular subspace of dimension $c(\log n+\log q)^3$, which exists by Witt's decomposition theorem. Note that any bijective linear map from $W$ to $W$ is an isometry, and is therefore subject to Witt's extension lemma.

By definition of a singularly transversal set, we can find $A_0\in T$ that maps the totally singular subspace $W$ onto itself, and that its restriction to this subspace is the matrix $(\bigoplus_{i=1}^rC_i)\oplus I$ for some arbitrary choices of basis on $W$, where $I$ is some identity matrix of suitable size.

In particular, $A_0$ is a P($r$)-matrix. Since $A_0\in T$, by choosing $T$ as in Corollary~\ref{Cor:ttransOSU}, $A_0$ has length bounded by $q^{cn(\log n+\log q)^3}$.

By using Lemma~\ref{Lem:pmatrix}, we can raise $A_0$ to a large power, and obtain a non-identity matrix $A_1$ of degree $\leq\frac{\deg(A_0)}{4}\leq\frac{n}{4}$, with eigenvalues either 1 or outside of $\F_q$. Since the order of $A_0$ is bounded by $q^n$ by Lemma~\ref{Lem:ordmatq}, the length of $A_1$ is bounded by $q^{cn(\log n+\log q)^3+n}$.

\emph{Inductive Step:}

Suppose we have obtained a non-identity matrix $A_j$ with eigenvalues either 1 or outside of $\F_q$, degree at most $\frac{n}{2^{j+1}}$, and length at most $q^{2cn(\log n+\log q)^3+j(n+2)}$. If $\deg A_j\leq 56+32c(\log n+\log q)^3$, then we stop. If not, then let us construct a non-identity matrix $A_{j+1}$ of even smaller degree.

First we shall transform $A_j$ into a P($r$)-matrix. Find a totally singular subspace $W_j$ of dimension $c(\log n+\log q)^3$ as in Lemma~\ref{Lem:ouspsubexi}. In particular, $W_j\oplus A_jW_j$ is a well-defined totally singular space. Let $T'$ be a singularly $2c(\log n+\log q)^3$-transversal set, then we can find $M_j\in T'$ that fixes $A_jW_j$, and restricts to a map from $W_j$ to $W_j$ as $\bigoplus_{i=1}^rC_i\oplus I$ for any arbitrary basis of $W_j$ and some identity matrix $I$ of suitable size.

Consider the commutator $M_jA_j^{-1}M_j^{-1}A_j$. Since $M_j$ fixes $A_jW_j$, we see that $M_jA_j^{-1}M_j^{-1}A_j$ restricted to $W_j$ is identical to $M_j$ restricted to $W_j$. 

In particular, $M_jA_j^{-1}M_j^{-1}A_j$ is a P($r$)-matrix, and it has degree at most $2\deg(A_j)$. Now we use Lemma~\ref{Lem:pmatrix} again, raising $M_jA_j^{-1}M_j^{-1}A_j$ to a large power, and we would obtain a matrix $A_{j+1}$ of degree at most $\frac{2\deg(A_j)}{4}$, with eigenvalue wither 1 or outside of $\F_q$.

Since $M_j\in T'$, by choosing $T'$ as in Corollary~\ref{Cor:ttransOSU}, $M_j$ have length bounded by $q^{2cn(\log n+\log q)^3}$. And since the order of $M_jA_j^{-1}M_j^{-1}A_j$ is bounded by $q^n$ by Lemma~\ref{Lem:ordmatq}, the length of $A_{j+1}$ is at most
\begin{align*}
q^n(2q^{2cn(\log n+\log q)^3}+2q^{2cn(\log n+\log q)^3+j(n+2)})\leq q^{2cn(\log n+\log q)^3+(j+1)(n+2)}.
\end{align*}

We repeat the above induction $\frac{\log n}{\log 2}-1$ times, or stop early if we hit degree $2c(\log n+\log q)^3$. The last $A_j$ we obtained is the desired matrix of small degree and small length.
\end{proof}

\section{The Conjugacy Expansion Lemmas}
\label{sec:conj}

In this section, we shall show that any small degree element will quickly generate the whole group with any symmetric generating set.

\begin{lem}
\label{ConjExp}
Let $S$ be any symmetric generating set for a subgroup $H$ of $\mathrm{GL}_n(\F_q)$. Let $A$ be any matrix in $H$ of degree $k$, and let $B$ be any matrix conjugate to $A$ in $H$. Then $B=MAM^{-1}$ for some $M\in H$ of length at most $q^{2nk}$.
\end{lem}
\begin{proof}
Since $A$ has degree $k$, we know $A=I+A'$ for some matrix $A'$ of rank $k$. So we can decompose $A'$ as a product $XY$ where $X$ is an $n$ by $k$ matrix of full rank and $Y$ is a $k$ by $n$ matrix of full rank. So $A=I+XY$.

Any conjugates of $A$ can similarly be expressed as $I+X'Y'$ where $X'$ is some $n$ by $k$ matrix of full rank, and $Y'$ is some $k$ by $n$ matrix of full rank. There are at most $q^{2nk}$ possibilities for the pair $(X',Y')$. So there are at most $q^{2nk}$ conjugates of $A$.

Now $H$ acts on the conjugacy class of $A$ in $H$ by left conjugation, and the corresponding Schreier graph must be connected. So the Schreier graph has diameter bounded by the number of vertices, i.e., $q^{2nk}$.
\end{proof}

\begin{thm}[(\cite{LS01})]
Let $G$ be $\mathrm{SL}_n(\F_q)$, $\mathrm{\Omega}_n(\F_q)$, $\mathrm{Sp}_n(\F_q)$, or $\mathrm{SU}_n(\F_q)$. Let $A\in G$ be an element of degree $k$ outside the center of $G$. Then every element of $G$ is a product of at most $O(\frac{n}{k})$ conjugates of $A$.
\end{thm}

\begin{prop}
\label{Prop:Conjfill}
Let $G$ be $\mathrm{SL}_n(\F_q)$, $\mathrm{\Omega}_n(\F_q)$, $\mathrm{Sp}_n(\F_q)$, or $\mathrm{SU}_n(\F_q)$. Let $S$ be any symmetric generating set for $G$. Suppose we have a non-trivial element $A\in G$ of length $d>0$ and degree $k<n$. Then the diameter of $G$ with respect to $S$ will be $O((2q^{2nk}+d)\frac{n}{k})$.
\end{prop}
\begin{proof}
For any $B$ conjugate to $A$, by the Lemma~\ref{ConjExp} above, $B=MAM^{-1}$ for some $M\in G$ of length at most $q^{2nk}$. So $B$ has length at most $2q^{2nk}+d$. So every conjugate of $A$ in $G$ has length bounded by $2q^{2nk}+d$.

Furthermore, since $A$ has degree $<n$ but non-trivial, it is not a scalar matrix. Then by \cite{CGGA}, $A$ is not in the center of $G$.

Now by the resulf of Liebeck and Shalev \cite{LS01}, we know that every element of $G$ can be written as a product of $O(\frac{n}{k})$ conjugates of $A$. So the whole group $G$ has a diameter bound of $O((2q^{2nk}+d)\frac{n}{k})$.
\end{proof}

\begin{cor}
The diameter of a finite simple group of Lie type of rank $n$ over $\F_q$ are at most $O(q^{O(n(\log n+\log q)^3)})$, independent of the choice of generating sets. The implied constants are absolute.
\end{cor}
\begin{proof}
Combine Proposition~\ref{Prop:Conjfill} with Proposition~\ref{Prop:Bdddeg} or Proposition~\ref{Prop:BdddegOSU}.
\end{proof}


\section{Implications on Spectral Gap and Mixing Time}

Given a group $G$ and its generating set $S$. Let $\G(G,S)$ be its Cayley graph, and let $A$ be the normalized adjacency matrix of the graph. Then $A$ has real eigenvalues $\lambda_1,...,\lambda_{|G|}$, ordered from the largest one to the smallest one. Then the \emph{spectral gap} of $\G(G,S)$ is $\lambda_1-\lambda_2$.

Let $\mu$ be the random distribution $\frac{1}{2}1_{\{e\}}+\frac{1}{2|S|}1_S$. Then a lazy random walk of length $k$ is the random outcome of the distribution $\mu^{(k)}=\mu\ast\mu\ast\mu\ast ... \ast\mu$. Using the definition of Helfgott, Seress and Zuk \cite{HSZ15}, the \emph{strong mixing time} of $\G(G,S)$ is the least number $k$ such that $\mu^{(k)}$ is at most $\frac{1}{2|G|}$ away from the uniform distribution on $\G(G,S)$, in the $\ell^\infty$ norm.

One can bound the spectral gap using a diameter bound. 

\begin{prop}[(\cite{DSC93}, Corollary 3.1)]
Given a finite group $G$ and a symmetric generating set $S$, let $\G$ be the Cayley graph. Then the spectral gap of the Cayley graph is bounded from below by $\frac{1}{(\diam\G)^2|S|}$
\end{prop}

In turn, one can bound the strong mixing time by the spectral gap.

\begin{prop}[(\cite{L96}, Theorem 5.1)]
Given a finite group $G$ and a symmetric generating set $S$, let $\G$ be the Cayley graph, and let $\lambda$ be the spectral gap. Then the strong mixing time of the Cayley graph is bounded by $O(\frac{\log |\G|}{\lambda})$.
\end{prop}

Then our main result implies the following corollary:

\begin{cor}
Let $G$ be a finite simple group of Lie type of rank $n$ over $\F_q$. The spectral gap of $\G(G,S)$ is bounded by $|S|^{-1}q^{-O(n(\log n+\log q)^3)}$, and the mixing time of $\G(G,S)$ is bounded by $|S|q^{O(n(\log n+\log q)^3)}$.
\end{cor}


\printbibliography
\end{document}